\documentclass[11pt]{amsart}

\usepackage{graphics}
\usepackage{color}
\usepackage[a4paper, margin=3cm]{geometry}

\usepackage{amssymb,enumerate}
\usepackage{amsmath}
\usepackage{bbm}           % used for blackboard 1
\usepackage{bm}
\usepackage{eso-pic,graphicx}
\usepackage{tikz}
\usepackage{cite}
\usepackage{esint}
\usepackage[colorlinks=true, pdfstartview=FitV, linkcolor=blue, citecolor=blue, urlcolor=blue]{hyperref}

\usepackage{graphicx}
\usepackage{pslatex}
\usepackage{amsmath,amsfonts}
\usepackage{rotating}
\usepackage{amssymb}
\usepackage{verbatim}
\usepackage{rotating}
\usepackage{mathrsfs}

\makeatletter
\def\Ddots{\mathinner{\mkern1mu\raise\p@
\vbox{\kern7\p@\hbox{.}}\mkern2mu
\raise4\p@\hbox{.}\mkern2mu\raise7\p@\hbox{.}\mkern1mu}}
\makeatother

\def\XXint#1#2#3{{\setbox0=\hbox{$#1{#2#3}{\int}$}
\vcenter{\hbox{$#2#3$}}\kern-.5\wd0}}

\begin{document}
\newtheorem{theorem}{Theorem}
\newtheorem{proposition}[theorem]{Proposition}
\newtheorem{conjecture}[theorem]{Conjecture}
\def\theconjecture{\unskip}
\newtheorem{corollary}[theorem]{Corollary}
\newtheorem{lemma}[theorem]{Lemma}
\newtheorem{claim}[theorem]{Claim}
\newtheorem{sublemma}[theorem]{Sublemma}
\newtheorem{observation}[theorem]{Observation}
\theoremstyle{definition}
\newtheorem{definition}{Definition}
\newtheorem{notation}[definition]{Notation}
\newtheorem{remark}[definition]{Remark}
\newtheorem{question}[definition]{Question}
\newtheorem{questions}[definition]{Questions}
\newtheorem{example}[definition]{Example}
\newtheorem{problem}[definition]{Problem}
\newtheorem{exercise}[definition]{Exercise}
 \newtheorem{thm}{Theorem}
 \newtheorem{cor}[thm]{Corollary}
 \newtheorem{lem}{Lemma}[section]
 \newtheorem{prop}[thm]{Proposition}
 \theoremstyle{definition}
 \newtheorem{dfn}[thm]{Definition}
 \theoremstyle{remark}
 \newtheorem{rem}{Remark}
 \newtheorem{ex}{Example}
 \numberwithin{equation}{section}
%%%%%%%%%%%%% Abbreviation of symbols %%%%%%%%%%%%%%%%
\def\C{\mathbb{C}}
\def\R{\mathbb{R}}
\def\Rn{{\mathbb{R}^n}}
\def\Rns{{\mathbb{R}^{n+1}}}
\def\Sn{{{S}^{n-1}}}
\def\M{\mathbb{M}}
\def\N{\mathbb{N}}
\def\Q{{\mathbb{Q}}}
\def\Z{\mathbb{Z}}
\def\F{\mathcal{F}}
\def\L{\mathcal{L}}
\def\S{\mathcal{S}}
\def\supp{\operatorname{supp}}
\def\essi{\operatornamewithlimits{ess\,inf}}
\def\esss{\operatornamewithlimits{ess\,sup}}
%%%%%%%%%%%%%%%%%%%%%%%%%%%%%%%%%%%%%%%%%%%%%%%%%%%%%

\numberwithin{equation}{section}
\numberwithin{thm}{section}
\numberwithin{theorem}{section}
\numberwithin{definition}{section}
\numberwithin{equation}{section}

\def\earrow{{\mathbf e}}
\def\rarrow{{\mathbf r}}
\def\uarrow{{\mathbf u}}
\def\varrow{{\mathbf V}}
\def\tpar{T_{\rm par}}
\def\apar{A_{\rm par}}

\def\reals{{\mathbb R}}
\def\torus{{\mathbb T}}
\def\scriptm{{\mathcal T}}
\def\heis{{\mathbb H}}
\def\integers{{\mathbb Z}}
\def\z{{\mathbb Z}}
\def\naturals{{\mathbb N}}
\def\complex{{\mathbb C}\/}
\def\distance{\operatorname{distance}\,}
\def\support{\operatorname{support}\,}
\def\dist{\operatorname{dist}\,}
\def\Span{\operatorname{span}\,}
\def\degree{\operatorname{degree}\,}
\def\kernel{\operatorname{kernel}\,}
\def\dim{\operatorname{dim}\,}
\def\codim{\operatorname{codim}}
\def\trace{\operatorname{trace\,}}
\def\Span{\operatorname{span}\,}
\def\dimension{\operatorname{dimension}\,}
\def\codimension{\operatorname{codimension}\,}
\def\nullspace{\scriptk}
\def\kernel{\operatorname{Ker}}
\def\ZZ{ {\mathbb Z} }
\def\p{\partial}
\def\rp{{ ^{-1} }}
\def\Re{\operatorname{Re\,} }
\def\Im{\operatorname{Im\,} }
\def\ov{\overline}
\def\eps{\varepsilon}
\def\lt{L^2}
\def\diver{\operatorname{div}}
\def\curl{\operatorname{curl}}
\def\etta{\eta}
\newcommand{\norm}[1]{ \|  #1 \|}
\def\expect{\mathbb E}
\def\bull{$\bullet$\ }

\def\blue{\color{blue}}
\def\red{\color{red}}

\def\xone{x_1}
\def\xtwo{x_2}
\def\xq{x_2+x_1^2}
\newcommand{\abr}[1]{ \langle  #1 \rangle}

\newcommand{\Norm}[1]{ \left\|  #1 \right\| }
\newcommand{\set}[1]{ \left\{ #1 \right\} }
\newcommand{\ifou}{\raisebox{-1ex}{$\check{}$}}
\def\one{\mathbf 1}
\def\whole{\mathbf V}
\newcommand{\modulo}[2]{[#1]_{#2}}
\def \essinf{\mathop{\rm essinf}}
\def\scriptf{{\mathcal F}}
\def\scriptg{{\mathcal G}}
\def\scriptm{{\mathcal M}}
\def\scriptb{{\mathcal B}}
\def\scriptc{{\mathcal C}}
\def\scriptt{{\mathcal T}}
\def\scripti{{\mathcal I}}
\def\scripte{{\mathcal E}}
\def\scriptv{{\mathcal V}}
\def\scriptw{{\mathcal W}}
\def\scriptu{{\mathcal U}}
\def\scriptS{{\mathcal S}}
\def\scripta{{\mathcal A}}
\def\scriptr{{\mathcal R}}
\def\scripto{{\mathcal O}}
\def\scripth{{\mathcal H}}
\def\scriptd{{\mathcal D}}
\def\scriptl{{\mathcal L}}
\def\scriptn{{\mathcal N}}
\def\scriptp{{\mathcal P}}
\def\scriptk{{\mathcal K}}
\def\frakv{{\mathfrak V}}
\def\C{\mathbb{C}}
\def\D{\mathcal{D}}
\def\R{\mathbb{R}}
\def\Rn{{\mathbb{R}^n}}
\def\rn{{\mathbb{R}^n}}
\def\Rm{{\mathbb{R}^{2n}}}
\def\r2n{{\mathbb{R}^{2n}}}
\def\Sn{{{S}^{n-1}}}
\def\M{\mathbb{M}}
\def\N{\mathbb{N}}
\def\Q{{\mathcal{Q}}}
\def\Z{\mathbb{Z}}
\def\F{\mathcal{F}}
\def\L{\mathcal{L}}
\def\G{\mathscr{G}}
\def\ch{\operatorname{ch}}
\def\supp{\operatorname{supp}}
\def\dist{\operatorname{dist}}
\def\essi{\operatornamewithlimits{ess\,inf}}
\def\esss{\operatornamewithlimits{ess\,sup}}
\def\dis{\displaystyle}
\def\dsum{\displaystyle\sum}
\def\dint{\displaystyle\int}
\def\dfrac{\displaystyle\frac}
\def\dsup{\displaystyle\sup}
\def\dlim{\displaystyle\lim}
\def\bom{\Omega}
\def\om{\omega}
\def\BMO{\rm BMO}
\def\CMO{\rm CMO}
\begin{comment}
\def\scriptx{{\mathcal X}}
\def\scriptj{{\mathcal J}}
\def\scriptr{{\mathcal R}}
\def\scriptS{{\mathcal S}}
\def\scripta{{\mathcal A}}
\def\scriptk{{\mathcal K}}
\def\scriptp{{\mathcal P}}
\def\frakg{{\mathfrak g}}
\def\frakG{{\mathfrak G}}
\def\boldn{\mathbf N}
\end{comment}

\author[S. Wang]{Shifen Wang}
\address{Shifen Wang:
	School of Mathematical Sciences \\
	Beijing Normal University \\
	Beijing 100875 \\
	People's Republic of China}
\email{wsfrong@mail.bnu.edu.cn}

\author[Q. Xue]{Qingying Xue$^{*}$}
\address{Qingying Xue:
	School of Mathematical Sciences \\
	Beijing Normal University \\
	Laboratory of Mathematics and Complex Systems \\
	Ministry of Education \\
	Beijing 100875 \\
	People's Republic of China}
\email{qyxue@bnu.edu.cn}
\author[C. Zhang]{Chunmei Zhang}
\address{Chunmei Zhang:
	School of Mathematical Sciences \\
	Beijing Normal University \\
	Beijing 100875 \\
	People's Republic of China}
\email{chmeizhang@mail.bnu.edu.cn}

\thanks{ The second author was supported partly by the National Key Research and Development Program of China (Grant No. 2020YFA0712900) and NSFC (No. 11871101).
\thanks{$^{*}$ Corresponding author, e-mail address: qyxue@bnu.edu.cn}}

\date{\today}
\title[ on weighted Compactness of Commutators ]
{\bf On weighted Compactness of Commutators of square function and semi-group maximal function associated to Schr\"odinger operators}

\begin{abstract}
In this paper, the object of our investigation is the following Littlewood-Paley square function $g$ associated with the Schr\"odinger operator $L=-\Delta+V$ defined by:
\begin{equation*}\label{1.1}
	g(f)(x)=\Big(\int_{0}^{\infty}\Big|\frac{d}{dt}e^{-tL}(f)(x)\Big|^2tdt\Big)^{1/2}.
\end{equation*}
where $\Delta$ is the laplacian operator on $\mathbb{R}^n$ and $V$is a nonnegative potential. We show that the commutators of $g$ are compact operators from $L^p(w)$ to $L^p(w)$  for $1<p<\infty$ if $b\in \CMO_\theta(\rho)$ and $w\in A_p^{\rho,\theta}$, where $\CMO_\theta(\rho)$ is the closure of $\mathcal{C}_c^\infty(\mathbb{R}^n)$ in the $\BMO_\theta(\rho)$ topology which is more larger than the classical $\CMO $ space  and $A_p^{\rho,\theta}$ is a weights class which is more larger than Muckenhoupt $A_p$ weight class. An extra weight condition in a privious weighted compactness result is removed for the commutators of the semi-group maximal function defined by $\mathcal{T}^*(f)(x)=\sup_{t>0}|e^{-tL}f(x)|.$

\end{abstract}\maketitle

\section{Introduction}

\medskip
This paper is devoted to study the weighted compactness of the commutators of the Littlewood-Paley square function $g$ associated with the Schr\"odinger operator $L=-\Delta+V$, where $\Delta$ is the well-known Laplacian operator and the function $V$ is a nonnegative potential enjoying certain reverse H\"{o}lder inequality in $\mathbb{R}^n (n\geq 3$).

It is well-known that the Schr\"{o}dinger operator $L$ plays a fundamental role in Harmonic analysis and PDE \cite{S}. In the past few decades, Great achievements have been made in the study of the operator $L$ and its associated operators such as Riesz transforms $\nabla L^{-1/2}$. It is worthy to pointing out that the
estimates of the solutions of $L$ and Riesz transforms $\nabla L^{-1/2}$ were given in \cite{F,S1,Z}. Great interests have been shown to investigate the properties of Hardy type spaces associated with $L$, such like, the Hardy space $H^p_L$  \cite{DZ2,DZ} for $0<p\leq 1$; the local Hardy spaces $h_L$ \cite{YYZ}. The semigroup maximal function $e^{-tL}(t>0)$ and the Riesz transforms $\nabla L^{-1/2}$ on  Heisenberg groups were further considered in \cite{LLLY}.

In order to state more known results, we need to introduce some notations and definitions. Let $V$ be a nonnnegative function defined on $\rn.$ For every ball $B(x,r)\subset\mathbb{R}^n$, we say $V$ is in the reverse H\"{o}lder class $RH_q$ ($q>n/2$), if there exists a constant $C>0$ such that $V$ enjoys the following  reverse H\"{o}lder inequality
\begin{equation*}
	\Big(\frac{1}{|B(x,r)|}\int_{B(x,r)}V^q(y)dy\Big)^{\frac{1}{q}}\leq C\Big(\frac{1}{|B(x,r)|}\int_{B(x,r)}V(y)dy\Big).
\end{equation*}

We need to give the definition of a class of more generalized $\BMO$ space associated with $\rho.$

\begin{definition}{\bf ($\BMO_\theta(\rho)(\rn)$ space, \cite{BHS})}. For $\theta>0$,
	we defined the class $\BMO_\theta(\rho)(\rn)$ of locally integrable functions $f$ such that
	\begin{equation}\label{1.5}
		\frac{1}{|B(x,r)|}\int_{B(x,r)}|f(y)-f_B|dy\leq C\Big(1+\frac{r}{\rho(x)}\Big)^{\theta },
	\end{equation}
	for all $x\in\mathbb{R}^n$ and $r>0$, where $f_B=\frac{1}{|B|}\int_Bf$ and  the function $\rho$ \cite{S1} associated with $V(x)\in RH_q $ is defined by
	\begin{equation}\label{1.3}
		\rho(x)=\sup_{r>0}\Big\{r:\frac{1}{r^{n-2}}\int_{B(x,r)}V(y)dy\leq1\Big\}.
	\end{equation} For a function $f\in \BMO_\theta(\rho)(\rn)$, the infimum of
	the constants in \eqref{1.5} is denoted by $\|f\|_{\BMO_\theta(\rho)}$.
	
\end{definition}
We now need tointroduce a more larger weights class than the classical $A_p$ weights class as follows.
\begin{definition}{\bf ($A_p^{\rho,\theta}$ weights class, \cite{BHS1})}.\label{def1.2}
	Let $w$ be a nonnegative, locally integrable function on
	$\mathbb{R}^n$. For $1<p<\infty$,  a weight $w$ is said in the class $A_p^{\rho,\theta}$ if there exists a positive constant $C$ such that for all balls $B=B(x,r)$, it holds that
	\begin{equation}\label{1.4}
		\Big(\frac{1}{|B|}\int_Bw(y)dy\Big)\Big(\frac{1}{|B|}\int_Bw(y)^{-1/(p-1)}dy\Big)^{p-1}\leq C\Big(1+\frac{r}{\rho(x)}\Big)^{\theta p}.
	\end{equation}
\end{definition}
\noindent $w\in A_1^{\rho,\theta}$ means that there exists a constant $C$ such that for all balls $B$, the following inequality holds
\begin{equation*}
	M^\theta_V(w)(x)\leq Cw(x),\ a.e. \ x\in\mathbb{R}^n,
\end{equation*}
where
$
M^\theta_V(f)(x)=\sup_{x\in B}\frac{1}{\Psi_\theta(B)|B|}\int_B|f(y)|dy
$ and
$\Psi_\theta(B)=(1+r/\rho(x))^\theta$ for $B=B(x,r)$.

\remark It is easy to see that $\BMO(\rn)\subset \BMO_\theta(\rho)(\rn)\subset \BMO_{\theta'}(\rho)(\rn)$ for $0<\theta<\theta'$, $A_p\subsetneq  A_p^{\rho,\theta}$ and the classes $A_p^{\rho,\theta}$ are increasing with $\theta$ for $1\leq p<\infty $. 
\vspace{0.2cm}

The Littlewood-Paley $g$ function and the semi-group maximal function associated to Schr\"odinger operator $L$ are defined by
\begin{equation}\label{1.1}
	g(f)(x)=\Big(\int_{0}^{\infty}\Big|\frac{d}{dt}e^{-tL}(f)(x)\Big|^2tdt\Big)^{1/2},
\end{equation}
and 
\begin{equation}\label{1.2}
	\mathcal{T}^*(f)(x)=\sup_{t>0}|e^{-tL}f(x)|=\sup_{t>0}\Big|\int_{\mathbb{R}^n}k_t(x,y)f(y)dy\Big|,
\end{equation}
where $k_t$ is the kernel of the operator $e^{-tL}$.
The commutators of $g$ and $\mathcal{T}^*$ with $b\in\BMO_\theta(\rho)(\rn)$ are defined by
\begin{equation}\label{1.6}
	g_b(f)(x)=\Big(\int_{0}^{\infty}\Big|\frac{d}{dt}e^{-tL}(b(x)-b(\cdot))(f)(x)\Big|^2tdt\Big)^{1/2};
\end{equation}
and
\begin{equation}\label{1.7}
	\mathcal{T}_b^*(f)(x)=\sup_{t>0}\Big|\int_{\mathbb{R}^n}k_t(x,y)(b(x)-b(y))f(y)dy\Big|.
\end{equation}

Recall that, in 2011, Bongioanni, Harboure and Salinas \cite{BHS} first considered the strong boundedness of the commutators of Riesz transforms related to $L$ with $\BMO_\theta(\rho)(\rn)$ functions. Weighted estimates were demonstrated in \cite{BHS1} for $g$ function and $	\mathcal{T}^*$ related to $L$ with  $A_p^{\rho,\theta}$ weights as follows:

\medskip
\quad\hspace{-20pt}{\bf Theorem A} (\cite{BHS1}). {\it  For $1<p<\infty$, the operator $g$ and $	\mathcal{T}^*$ are bounded on $L^p(w)$ when $w\in A_p^{\rho,\theta}$.}
\vspace{0.2cm}

In another paper, Tang \cite{T} considered the weighted norm inequalities for commutators of $g$ function and the maximal function  $\mathcal{T}^*$ associated with Schr\"odinger operators $L.$

\medskip
\quad\hspace{-20pt}{\bf Theorem B} (\cite{T}). {\it  Let $1<p<\infty$, $w\in A_p^{\rho,\theta}$ and $b\in\BMO_\theta(\rho)$, then there exists a constant $C$ such that
	\begin{equation*}
		\|g_b(f)\|_{L^p(w)}\leq C\|b\|_{\BMO_\theta(\rho)}\|f\|_{L^p(w)},\quad
		\|{T}_b^*(f)\|_{L^p(w)}\leq C\|b\|_{\BMO_\theta(\rho)}\|f\|_{L^p(w)}.
\end{equation*}}
\smallskip

As was said before, this paper is devoted to studying the weighted compactness for commutators of
Littlewood-Paley function related to Schr\"odinger operator $L$. Recall that, a linear or sublinear operator $T$ is said to be a compact operator from a Banach space $X$ to a Banach space $Y$, if the image under $T$ of any bounded subset in $X$ is a relatively compact subset in $Y.$ Given a locally integrable function $b$,
the commutator $[b,T]$ of $T$ is given by $[b,T](f)(x)=bTf(x)-T(bf)(x).$ 

We now present some known result about the compactness of some classical operators.
In 1975, Corders \cite{C} first showed that the commutators of the Fourier multipliers and pseudodifferential operators are compact on $L^p(\rn)$ provided $b\in{\rm CMO}(\mathbb{R}^n)$, where ${\rm CMO}(\mathbb{R}^n)$ is the closure of $C_c^{\infty}$ in the $\BMO(\mathbb{R}^n)$ topology. In 1978, Uchiyama \cite{Uch}
surprisingly proved that the commutator of singular intergral opertors with $ {\rm Lip}_1({\rm S}^{n-1})$ kernels
is compact on $L^p(\mathbb{R}^n)$ for $1<p<\infty$ if and only if
$b\in{\rm CMO}(\mathbb{R}^n)$.
Since then, the study on the compactness of commutators of different operators
has attracted much more attention. We refer the reader to \cite{KL1,KL2,W,CD1,CD2,CDW,LWX} for the compactness of commutators of
linear operators, 
\cite{BT,BDMT,XYY,TX,TXYY} for multilinear operators and \cite{COY1,H1,H2,H3} for extrapolation theory of compactness. In particular, in a recent paper , the result of Corders \cite{C} was extended to the bilinear case and the the condition $b\in \CMO(\rn) $ was improved to a more bigger space which is called $\rm {XMO}(\rn)$.

Note that these compactness results are all concerned with the space $\text {CMO}(\mathbb{R}^n)$. However, Theorem B shows that the $L^p$ boundedness holds for more larger space $\text{BMO}_\theta(\rho)(\mathbb{R}^n)$, rather than $\text{BMO}(\mathbb{R}^n)$ and the weights class $A_p^{\rho,\theta}$ is more larger than $A_p$ weights class. Let $\text{ CMO}_\theta(\rho)(\mathbb{R}^n)$ be the closure of $\mathcal{C}_c^\infty(\mathbb{R}^n)$ in the $\BMO_\theta(\rho)(\mathbb{R}^n)$ topology.  Then, it is quite natural to ask the following question:

\begin{question}
	Whether the operator $g_b$ is compact from  $L^p(w)$ to $L^p(w)$ when $w\in A_p^{\rho,\theta}$ and $b$ belongs to the space $\text{ CMO}_\theta(\rho)(\mathbb{R}^n)$?
\end{question}

One may ask the same question for the operator $\mathcal{T}^*_b$. Very recently,  we obtained the following weighted compactness for  $\mathcal{T}^*_b$ in \cite{WX}. However, an extra weights condition needs to be assumed. 
\medskip
\quad\hspace{-20pt}{\bf Theorem C} (\cite{WX}). {
	Let $1<p<\infty$, $w\in A_p^{\rho,\theta}$ and $b\in\CMO_\theta(\rho)(\rn)$. If $w$ satisfies the following condition
	\begin{equation}\label{1.8}
		\lim_{A\rightarrow+\infty}A^{-np+n}\int_{|x|>1}\frac{w(Ax)}{{|x|}^{np}}dx=0,
	\end{equation}
	then the operator $\mathcal{T}^*_b$ is a compact operator from $L^p(w)$ to $L^p(w)$.}

Therefore, another question arises naturally:
\begin{question}
	Is it possible to remove condition \eqref{1.8} from Theorem C ?
\end{question}

The main purpose of this paper is to give firm answers to the above questions. Our results are as follows:

\medskip
\begin{theorem}\label{thm1.1}
	Let $1<p<\infty$. If $w\in A_p^{\rho,\theta}$ and $b\in\rm CMO_\theta(\rho)(\mathbb{R}^n)$,
	then the operator $g_b$ is a compact operator from $L^p(w)$ to $L^p(w)$.
\end{theorem}
\begin{theorem}\label{thm1.2}
	Let $1<p<\infty$. If $w\in A_p^{\rho,\theta}$ and $b\in\rm CMO_\theta(\rho)(\mathbb{R}^n)$,
	then the operator $\mathcal{T}^*_b$ defined by \eqref{1.7} is a compact operator from $L^p(w)$ to $L^p(w)$.
\end{theorem}

The organization of this paper is as follows. In section \ref{S2}, we prepare some definitions and preliminary lemmas, which play a fundamental role in our analysis. Section \ref{S3} is devoted to give the proof of Theorem \ref{thm1.1} via smoothness truncated techniques, which is different from what have been used in \cite{WX}. The domain of integration will be divided into several cases, some cases are combinable, but various subcases also arise, which increases the difficulty we need to deal with. 

Throughout the paper,  the letter $C$ or $c$, sometimes with certain parameters,
will stand for positive constants not necessarily the same one at each occurrence,
but are independent of the essential variables.  $A\sim B$ means that there exists
constants $C_1>0$ and $C_2>0$ such that $C_2B\leq A\leq C_1B$.

\bigskip

\section{Preliminaries }\label{S2}
\medskip
We first recall some notation and lemmas which will be used later. Given
a Lebesgue measurable set $E\subset\mathbb{R}^n$, $|E|$ will denote the Lebesgue measure of $E$. If $B = B(x, r)$ is a ball in $\mathbb{R}^n$ and $\lambda$ is a real number, then $\lambda B$ shall stand for the ball with the same center as $B$ and radiu $\lambda$ times that of $B$. A weight $w$ is a non-negative measurable function on $\mathbb{R}^n$. The measure associated with $w$ is the set function given by $w(E)=\int_E wdx$. For $0<p<\infty$ we denote by $L^p(w)$ the space of all Lebesgue measurable function $f(x)$ such that
$$\|f\|_{L^p(w)}=\Big(\int_{\mathbb{R}^n}|f(x)|^pw(x)dx|\Big)^{1/p}.$$

The auxiliary function $\rho$ defined in (\ref{1.3}) enjoys the following property.
\begin{lemma}\label{lem2.1}	{\rm (\rm\cite{S1})}.
	There exist $k_0\geq1$ and $C>0$ such that for all $x,\ y\in\rn$,
	\begin{equation*}
	C^{-1}\rho(x)\Big(1+\frac{|x-y|}{\rho(x)}\Big)^{-k_0}\leq\rho(y)\leq C\rho(x)(1+\frac{|x-y|}{\rho(x)}\Big)^{\frac{k_0}{k_0+1}}.
	\end{equation*}
	In particular, $\rho(x)\sim\rho(y)$ if $|x-y|<C\rho(x)$.
\end{lemma}

It is known that
$A_p^{\rho,\theta}$ weights class has some properties analogy to $A_p$ weights class.
\begin{lemma}\label{lem2.2}	{\rm (\rm \cite{BHS1,T1})}.
	Let $1<p<\infty$ and $w\in A_p^{\rho,\infty}=\bigcup_{\theta\geq0}A_p^{\rho,\theta}$. Then
	\begin{enumerate}[{\rm (i)}]
		\item If $1\leq p_1<p_2<\infty$, then $A_{p_1}^{\rho,\theta}\subset A_{p_2}^{\rho,\theta}.$
		\item $w\in A_p^{\rho,\theta}$ if and only if $w^{-\frac{1}{p-1}}\in A_{p'}^{\rho,\theta}$, where $1/p+1/p'=1$.
		\item If  $w\in A_p^{\rho,\infty}$, $1<p<\infty$,  then there exists $\epsilon>0$ such that $w\in A_{p-\epsilon}^{\rho,\infty}$.
		\item if $w\in A_p^{\rho,\theta}$, $1\leq p<\infty$, then for any cube $Q$, we have
		$$\frac{1}{\Psi_\theta(Q)|Q|}\int_Q|f(y)|dy\leq C\Big(\frac{1}{w(5Q)}\int_Q|f|^pw(y)dy\Big)^{1/p},$$
		where $\Psi_\theta$ is the same function appeared in Definition \ref{def1.2}. 
	\end{enumerate}
\end{lemma}

It should be pointed out that {\rm(iii)} of Lemma \ref{lem2.2} was proved by Bongioanni et al. \cite{BHS1} and  {\rm(iv)} of Lemma \ref{lem2.2} was proved by Tang \cite{T1}.

\medskip

For maximal operator $M_{V}^\theta$ in Definition \ref{def1.2}, the following result holds:

\begin{lemma}\label{lem2.3}	{\rm (\rm\cite{T1})}.
	Let $1<p<\infty$ and suppose that $w\in A_p^{\rho,\theta}$. If $p<p_1<\infty$,
	then
	\begin{equation*}
	\int_{\mathbb{R}^n}|M_V^\theta f(x)|^{p_1}w(x)dx\leq C_p\int_{\mathbb{R}^n}|f(x)|^{p_1}w(x)dx.
	\end{equation*}
\end{lemma}

By Lemma \ref{lem2.3},  $M_V^\theta$ may not be bounded on $L^p(w)$ for all $w\in A_p^{\rho,\theta}$ and $1<p<\infty$. A variant maximal operator $M_{V,\eta}$ was introduced by Tang (\cite{T1}) as follows:
\begin{equation*}
M_{V,\eta} f(x)=\sup_{x\in B}\frac{1}{(\Psi_\theta(B))^\eta|B|}\int_B|f(y)|dy,\ \ \ 0<\eta<\infty.
\end{equation*}

This variant maximal function $M_{V,\eta} $ then enjoys the following $L^p$ boundedness.
\begin{lemma}\label{lem2.4}	{\rm (\rm\cite{T1})}.
	Let $1<p<\infty$, $p'=p/(p-1)$ and suppose that $w\in A_p^{\rho,\theta}$. Then there  exists a constant $C>0$ such that
	\begin{equation*}
	\|M_{V,p'}f\|_{L^p(w)}\leq C\|f\|_{L^p(w)}.
	\end{equation*}
\end{lemma}

\medskip
Let $q_t$ be the kernel of $\frac{d}{dt}e^{-tL}$. Now, we list some pointwise estimates of $q_t$.
\begin{lemma}\label{lem2.5}	{\rm (\rm \cite{DGTZ})}
	For every $N$, there are constants $c$ and $C_N$ such that
	\begin{equation*}
	|q_t(x,y)|\leq \frac{ C_N}{t^{n/2+1}}e^{-\frac{c|x-y|^2}{t}}\Big(1+\frac{\sqrt{t}}{\rho(x)}+\frac{\sqrt{t}}{\rho(y)}\Big)^{-N}.
	\end{equation*}
\end{lemma}

\begin{lemma}\label{lem2.6}	{\rm (\rm \cite{DGTZ})}
	There exist $0 < \delta <1$ and $c>0$ such that for every $N>0$ there is a constant
	$C_N>0$ so that, for all $|h|\leq\sqrt{t}$,
	\begin{equation*}
	|q_t(x+h)-q_t(x,y)|\leq\frac{ C_N}{t^{n/2+1}}\Big(\frac{|h|}{\sqrt{t}}\Big)^\delta e^{-\frac{c|x-y|^2}{t}}\Big(1+\frac{\sqrt{t}}{\rho(x)}+\frac{\sqrt{t}}{\rho(y)}\Big)^{-N}.
	\end{equation*}
\end{lemma}

Let $X$ and $Y$ be Banach spaces and suppose $T$ and $\{T_n\}$ be operators from $X$ to $Y$. Then, we need the limiting property of compact operators. 
\begin{lemma}\label{lem2.7}{\rm (\rm \cite[p.\,278, Theorem(iii)]{Y})}
	Let a sequence $\{T_n\}$ of compact operators converge to an operator T in sense of the uniform operator toplogy, i.e., $\lim\limits_{n\rightarrow\infty}\|T-T_n\|=0$. Then $T$ is also compact.
\end{lemma}
\medskip
We end this section by introducing the general weighted version of Frechet-Kolmogorov
theorems, which was given by Xue, Yabuta and Yan in \cite{XYY}.
\begin{lemma}\label{lem2.8}	{\rm (\rm \cite{XYY})}. Let $w$ be a weight on $\mathbb{R}^n$. Assume that $w^{-1/(p_0-1)}$ is also a weight on $\mathbb{R}^n$ for some $p_0>1$. Let $0<p<\infty$ and $\mathcal{F}$ be a subset in $L^p(w)$, then $\mathcal{F}$ is sequentially compact in $L^p(w)$ if the following three conditions are satisfied:
	\begin{enumerate}[{\rm (i)}]	
		\item $\mathcal{F}$ is bounded, i.e.,
		$\sup\limits_{f\in\mathcal{F}}\|f\|_{L^p(w)}<\infty$;
		\item $\mathcal{F}$ uniformly vanishes at infinity, i.e.,
		\begin{equation*}
		\lim\limits_{N\rightarrow\infty}\sup\limits_{f\in\mathcal{F}}\int_{|x|>N}|f(x)|^pw(x)dx=0;
		\end{equation*}
		\item $\mathcal{F}$ is uniformly equicontinuous, i.e.,
		\begin{equation*}
		\lim\limits_{|h|\rightarrow0}\sup\limits_{f\in\mathcal{F}}\int_{\mathbb{R}^n}|f(\cdot+h)-f(\cdot)|^pw(x)dx=0.
		\end{equation*}
	\end{enumerate}
\end{lemma}

\bigskip

\section{Proof of Theorem \ref{thm1.1} and \ref{thm1.2} }\label{S3}
\medskip

\begin{proof}[Proof of Theorem \ref{thm1.1}] We shall prove Theorem \ref{thm1.1} via smoothness truncated techniques, which is different from that used in \cite{WX}.  Let $\varphi\in C^{\infty}([0,\infty))$ satisfy
	\begin{eqnarray}\label{3.1}
	0\leq\varphi\leq1\ \ \ and\ \ \
	\varphi(t)=
	\begin{cases}
	1,          &t\in[0,1],\\
	0,          &t\in[2,\infty).
	\end{cases}
	\end{eqnarray}
	For any $\gamma>0$, let
	\begin{equation}\label{3.2}
	q_{t,\gamma}(x,y)=q_t(x,y)(1-\varphi(\frac{\sqrt{t}}{\gamma})),
	\end{equation}
where $q_t$ is the kernel of $\frac{d}{dt}e^{-tL}$. Note that this truncated function essentially just truncates $t$ and does not contribute anything to $x$ and $y$. With the kernel $q_{t,\gamma}$, we define
	\begin{equation}\label{3.3}
	g_\gamma f(x)=\Big(\int_{0}^{\infty}\Big|\int_{\rn}q_{t,\gamma}(x,y)f(y)dy\Big|^2tdt\Big)^{1/2}
	\end{equation}
	and
	\begin{equation}\label{3.4}
	g_{b,\gamma} f(x)=\Big(\int_{0}^{\infty}\Big|\int_{\rn}q_{t,\gamma}(x,y)(b(x)-b(y))f(y)dy\Big|^2tdt\Big)^{1/2}.
	\end{equation}
Then, for any $|h|\leq\sqrt{t}$, by Lemma \ref{lem2.5} and Lemma \ref{lem2.6}, we have
\begin{align}\label{3.5'}
\begin{split}
|q_{t,\gamma}(x,y)|\leq |q_t(x,y)|
\leq \frac{ C_N}{t^{n/2+1}}e^{-\frac{c|x-y|^2}{t}}\Big(1+\frac{\sqrt{t}}{\rho(x)}+\frac{\sqrt{t}}{\rho(y)}\Big)^{-N}
\end{split}
\end{align}
and
\begin{align}\label{3.5}
\begin{split}
|q_{t,\gamma}(x+h,y)-q_{t,\gamma}(x,y)|
&\leq |q_{t}(x+h,y)-q_{t}(x,y)|\\ 
&\leq \frac{ C_N}{t^{n/2+1}}\Big(\frac{|h|}{\sqrt{t}}\Big)^\delta e^{-\frac{c|x-y|^2}{t}}\Big(1+\frac{\sqrt{t}}{\rho(x)}+\frac{\sqrt{t}}{\rho(y)}\Big)^{-N}.
\end{split}
\end{align}

For any $b\in\mathcal{C}_c^\infty(\rn)$ and $\gamma,\theta,\ \eta>0$, by \eqref{3.2}, \eqref{3.4} and Lemma \ref{lem2.5} with $N=\theta\eta$, one has
\begin{align}\label{3.6}
\begin{split}
&|g_{b,\gamma} f(x)-g_{b} f(x)|\\
&\leq C\displaystyle\Big(\int_{0}^{4\gamma^2}\Big(\int_{\mathbb{R}^n}t^{-\frac{n}{2}-1}e^{-\frac{c|x-y|^2}{t}}\Big(1+\frac{\sqrt{t}}{\rho(x)}\Big)^{-\theta\eta}|b(x)-b(y)||f(y)|dy\Big)^2tdt\Big)^{1/2}\\
&\leq C\displaystyle\Big(\int_{0}^{4\gamma^2}\Big(\int_{|x-y|<\sqrt{t}}t^{-\frac{n}{2}-1}e^{-\frac{c|x-y|^2}{t}}\Big(1+\frac{\sqrt{t}}{\rho(x)}\Big)^{-\theta\eta}|b(x)-b(y)||f(y)|dy\Big)^2tdt\Big)^{1/2}\\
&\quad+C\displaystyle\Big(\int_{0}^{4\gamma^2}\Big(\int_{|x-y|\geq\sqrt{t}}t^{-\frac{n}{2}-1}e^{-\frac{c|x-y|^2}{t}}\Big(1+\frac{\sqrt{t}}{\rho(x)}\Big)^{-\theta\eta}|b(x)-b(y)||f(y)|dy\Big)^2tdt\Big)^{1/2}\\
&=:I_1+I_2.
\end{split}
\end{align}
For $I_1$, we have
\begin{align}\label{3.7}
\begin{split}
I_1&\leq C \displaystyle\Big(\int_{0}^{4\gamma^2}\Big(t^{-\frac{n}{2}-1}\sqrt{t}\Big(1+\frac{\sqrt{t}}{\rho(x)}\Big)^{-\theta\eta}\int_{|x-y|<\sqrt{t}}|f(y)|dy\Big)^2tdt\Big)^{1/2}\\
&\leq C \displaystyle\Big(\int_{0}^{4\gamma^2}\Big(t^{-\frac{n}{2}}\Big(1+\frac{\sqrt{t}}{\rho(x)}\Big)^{-\theta\eta}\int_{|x-y|<\sqrt{t}}|f(y)|dy\Big)^2dt\Big)^{1/2}\\
&\leq C\Big(\int_{0}^{4\gamma^2}dt\Big)^{1/2} M_{V,\eta}f(x)\\
&\leq C\gamma M_{V,\eta}f(x).
\end{split}
\end{align}
For $I_2$, using the estimate $e^{-s}\leq\frac{C}{s^{M/2}}$with $M>n+1+\theta\eta$ and splitting to annuli, it follows that
\begin{align}\label{3.8}
\begin{split}
I_2&\leq C\displaystyle\Big(\int_{0}^{4\gamma^2}\Big( t^{\frac{M-n}{2}-1}\Big(1+\frac{\sqrt{t}}{\rho(x)}\Big)^{-\theta\eta}\int_{|x-y|\geq\sqrt{t}}\frac{|b(x+h)-b(y)||f(y)|}{|x-y|^M}dy\Big)^2tdt\Big)^{1/2}\\
&\leq C\displaystyle \Big(\int_{0}^{4\gamma^2}\Big(t^{\frac{M-n}{2}-1}\Big(1+\frac{\sqrt{t}}{\rho(x)}\Big)^{-\theta\eta}\sum_{k=1}^{\infty}\frac{2^k\sqrt{t}}{(2^k\sqrt{t})^M}\int_{|x-y|\sim2^k\sqrt{t}}|f(y)|dy\Big)^2tdt\Big)^{1/2}\\
&\leq C\displaystyle \Big(\int_{0}^{4\gamma^2}\Big(t^{-\frac{n}{2}}\Big(1+\frac{2^k\sqrt{t}}{\rho(x)}\Big)^{-\theta\eta}\sum_{k=1}^{\infty}\frac{2^{-k(M-n-1-\theta\eta)}}{(2^k\sqrt{t})^M}\int_{|x-y|<2^k\sqrt{t}}|f(y)|dy\Big)^2dt\Big)^{1/2}\\
&\leq C\gamma M_{V,\eta}f(x).
\end{split}
\end{align}

Combing \eqref{3.8} with \eqref{3.6} and \eqref{3.7} may lead to
\begin{equation*}
|g_{b,\gamma} (f)(x)-g_{b}(f)(x)|\leq C\gamma M_{V,\eta}f(x).
\end{equation*}
Then, Lemma \ref{lem2.4} with $p'\leq\eta<\infty$ gives that
\begin{equation}\label{3.9}
\|g_{b,\gamma}(f)-g_{b}(f)\|_{L^p(w)}\leq C\gamma\|f\|_{L^p(w)},
\end{equation}
which implies that
\begin{equation}\label{3.10}
\lim_{\gamma\rightarrow 0}\|g_{b,\gamma}(f)-g_b(f)\|_{L^p(w)\rightarrow L^p(w)}=0.
\end{equation}

On the other hand, if $b\in\CMO_\theta(\rho)(\rn)$, then for any $\epsilon>0$, there exists $b_\epsilon\in\mathcal{C}_c^\infty(\rn)$ such that $\|b-b_\epsilon\|_{\BMO_\theta(\rho)}<\epsilon$. Therefore
\begin{equation*}
\|g_b(f)-g_{b_\epsilon} (f)\|_{L^p(w)}\leq\|g_{b-b_\epsilon}(f)\|_{L^p(w)}\leq C\|b-b_\epsilon\|_{\BMO_\theta(\rho)}\|f\|_{L^p(w)}\leq C\epsilon.
\end{equation*}
Thus, to prove $g_b$ is compact on
$L^p(w)$ for any $b\in\CMO_\theta(\rho)$, it suffices to prove that $g_b$ is compact on $L^p(w)$ for any $b\in\mathcal{C}_c^\infty(\mathbb{R}^n)$. By \eqref{3.10} and Lemma \ref{lem2.7}, it suffices to show that $g_{b,\gamma}$ is compact for any $b\in\mathcal{C}_c^\infty(\mathbb{R}^n)$ when $\gamma>0$ is small enough. To this end, for arbitrary bounded set $F$ in $L^p(w)$, let
$$\mathcal{F}=\{g_{b,\gamma}f:f\in F\}.$$
Then, we need to show that for $b\in \mathcal{C}_c^\infty(\mathbb{R}^n)$, $\mathcal{F}$ satisfies the conditions$
\rm (i)$-$\rm(iii)$ of Lemma \ref{lem2.8}.

From the definition of $q_{t,\gamma}$, we know that $g_\gamma f(x)\leq gf(x)$ and $g_{b,\gamma} f(x)\leq g_bf(x)$. Hence, Theorem A and Theorem B still hold for $g_{\gamma}$ and $g_{b,\gamma}$ respectively. Then
\begin{equation*}
\sup\limits_{f\in F}\|g_{b,\gamma}f\|_{L^p(w)}\leq C\sup\limits_{f\in F}\|f\|_{L^p(w)}\leq C,
\end{equation*}
which yields the fact that the set $\mathcal{F}$ is bounded.

To prove that for any $w\in A_p^{\rho,\theta}$, (ii) of Lemma \ref{lem2.8}
\begin{equation}\label{3.11}
\lim_{A\rightarrow\infty}\int_{|x|>A}|g_{b,\gamma}f(x)|^pw(x)dx=0,
\end{equation}
holds for $f\in F$, we need the following claim

\textbf{Claim.} For any $N>0$, there is a constant $C_N>0$ so that 
\begin{equation}\label{3.12}
\Big(\int_{0}^{\infty}|q_{t,\gamma}(x,y)|^2tdt\Big)^{\frac{1}{2}}\leq \frac{C_N}{|x-y|^{n}}\Big(1+\frac{|x-y|}{\rho(x)}\Big)^{-N}.
\end{equation}
In fact, by \eqref{3.5'}, we may decompose the left side of \eqref{3.12} as
\begin{align*}
\displaystyle\Big(\int_{0}^{\infty}|q_{t,\gamma}(x,y)|^2tdt\Big)^{\frac{1}{2}}&\leq\displaystyle C_N\Big(\int_{0}^{|x-y|^2}t^{-n-2}e^{-\frac{c|x-y|^2}{t}}\Big(1+\frac{\sqrt{t}}{\rho(x)}\Big)^{-2N}tdt\Big)^{\frac{1}{2}}\\
&\qquad+\displaystyle C_N\Big(\int_{|x-y|^2}^{\infty}t^{-n-2}e^{-\frac{c|x-y|^2}{t}}\Big(1+\frac{\sqrt{t}}{\rho(x)}\Big)^{-2N}tdt\Big)^{\frac{1}{2}}\\
&=:J_1+J_2.
\end{align*}

In order to estimate $J_1$ and $J_2$,  we need the following inequality: for any $M>0$, there exists a constant $C>0$, such that
\begin{equation}\label{3.13}
e^{-\frac{c|x-y|^2}{t}}\leq C\frac{t^{\frac{M}{2}}}{|x-y|^{M}}.
\end{equation}
For $J_1$, using \eqref{3.13} with $M=n+N+1$,  we obtain
\begin{align*}
J_1&\leq\displaystyle C_N\Big(\int_{0}^{|x-y|^2}t^{-n-2}\Big(\frac{t}{|x-y|^2}\Big)^{n+N+1}\Big(1+\frac{\sqrt{t}}{\rho(x)}\Big)^{-2N}tdt\Big)^{\frac{1}{2}}\\
&\leq C_N\Big(\int_{0}^{|x-y|^2}\frac{1}{|x-y|^{2(n+1)}}\Big(\frac{\sqrt{t}}{|x-y|}\Big)^{2N}\Big(1+\frac{\sqrt{t}}{\rho(x)}\Big)^{-2N}dt\Big)^{\frac{1}{2}}\\
&\leq \frac{C_N}{|x-y|^{n+1}}\Big(\int_{0}^{|x-y|^2}\Big(\frac{\sqrt{t}+\rho(x)}{|x-y|+\rho(x)}\Big)^{2N}\Big(\frac{\sqrt{t}+\rho(x)}{\rho(x)}\Big)^{-2N}dt\Big)^{\frac{1}{2}}\\
&\leq \frac{C_N}{|x-y|^{n}}\Big(1+\frac{|x-y|}{\rho(x)}\Big)^{-N}.
\end{align*}
For $J_2$, using \eqref{3.13} with $M=2(n-1)$,  we get
\begin{align*}
J_2&\leq\displaystyle C_N\Big(\int_{|x-y|^2}^{\infty}t^{-n-1}\Big(\frac{t}{|x-y|^2}\Big)^{n-1}\Big(1+\frac{|x-y|}{\rho(x)}\Big)^{-2N}dt\Big)^{\frac{1}{2}}\\
&\leq \frac{C_N}{|x-y|^{n-1}}\Big(1+\frac{|x-y|}{\rho(x)}\Big)^{-N}\Big(\int_{|x-y|^2}^{\infty}t^{-2}dt\Big)^{\frac{1}{2}}\\
&\leq \frac{C_N}{|x-y|^{n}}\Big(1+\frac{|x-y|}{\rho(x)}\Big)^{-N}.
\end{align*}
Therefore, \eqref{3.12} holds, and this finishes the proof of this claim.

Now we are in the position to use the above claim to prove \eqref{3.11}.
Assume $b\in\mathcal{C}_c^\infty(\mathbb{R}^n)$ and $\supp(b)\subset B(0,R)$, where $B(0,R)$ is the ball of radius $R$ center at origin in $\mathbb{R}^n$. For any $|x| > A > 2R$, $w \in A_p^{\rho,\theta}$, $1<p<\infty$  and $f\in F$.  Then, by the inequality \eqref{3.12}, we have
\begin{equation*}
\begin{split}
|g_{b,\gamma}(f)(x)|&\leq\displaystyle \int_{|y|<R}\Big(\int_{0}^{\infty}|q_{t,\gamma}(x,y)|^2tdt\Big)^{\frac{1}{2}}|b(y)f(y)|dy\\
&\leq C_N\|b\|_\infty\displaystyle \int_{|y|<R}\Big(1+\frac{|x-y|}{\rho(x)}\Big)^{-N}\frac{|f(y)|}{|x-y|^{n}}dy\\
&\leq \frac{C_N\|b\|_\infty}{|x|^{n}}\frac{\|f\|_{L^p(w)}}{(1+1/2|x|m_V(x))^N}\displaystyle \Big(\int_{|y|<R}w^{-p'/p}(y)dy\Big)^{1/p'}.
\end{split}
\end{equation*}
By Lemma \ref{lem2.1}, there exist $k_0\geq 1$ and $C_0>0$ such that
$$m_V(x)\geq C_0m_V(0)(1+|x|m_V(0))^{-\frac{k_0}{k_0+1}}.$$
Therefore, it holds that
\begin{equation*}
\begin{split}
\frac{1}{1+1/2|x|m_V(x)}&\leq\frac{C}{1+|x|m_V(0)(1+|x|m_V(0))^{-\frac{k_0}{k_0+1}}}\\
&\leq\frac{C}{(1+|x|m_V(0))(1+|x|m_V(0))^{-\frac{k_0}{k_0+1}}}\\&=\frac{C}{(1+|x|/\rho(0))^{\frac{1}{k_0+1}}}
\end{split}
\end{equation*}
which leads to
\begin{equation*}
\begin{array}{ll}
{}&\displaystyle\Big(\int_{|x|>A}|g_{b,\gamma}f(x)|^pw(x)dx\Big)^{1/p}\\
&\leq C\displaystyle\sum_{j=0}^{\infty}\frac{1}{(2^jA)^{n}}\frac{1}{(1+2^jA/\rho(0))^{\frac{N}{k_0+1}}}\Big(\int_{2^jA<|x|<2^{j+1}A}w(x)dx\Big)^{1/p}\displaystyle\Big(\int_{|y|<R}w^{-p'/p}(y)dy\Big)^{1/p'}\\
&\leq C\displaystyle\sum_{j=0}^{\infty}\frac{1}{(2^jA)^{n}}\frac{1}{(1+2^jA/\rho(0))^{\frac{N}{k_0+1}}}\Big(\int_{B(0,2^{j+1}A)}w(x)dx\Big)^{1/p}\displaystyle\Big(\int_{B(0,R)}w^{-p'/p}(y)dy\Big)^{1/p'}.
\end{array}
\end{equation*}

Let $Q=B(0,2^{j+1}A)$, $E=B(0,R)$. Since $w\in A_{p}^{\rho,\theta}$, then, by Lemma \ref{lem2.2}, we can get
\begin{equation*}
\begin{array}{ll}
w(5Q)&\leq C\displaystyle\Big(\frac{\Psi_\theta(Q)|Q|}{|E|}\Big)^pw(E)\\
&\leq C\displaystyle w(B(0,R))\Big(\frac{(1+2^{j+1}A/\rho(0))^\theta(2^{j+1}A)^n}{R^n}\Big)^p\\
&\leq C\displaystyle w(B(0,R))(1+2^{j+1}A/\rho(0))^{\theta p}\frac{(2^{j+1}A)^{np}}{R^{np}}
\end{array}
\end{equation*}
Taking $N>(k_0+1)\theta$, we may continue to estimate $\Big(\int_{|x|>A}|g_{b,\gamma}f(x)|^pw(x)dx\Big)^{1/p}$ in the way that
\begin{equation*}
\begin{array}{ll}
&\quad\displaystyle\Big(\int_{|x|>A}|g_{b,\gamma}f(x)f(x)|^pw(x)dx\Big)^{1/p}\\
&\quad\leq C\displaystyle\sum_{j=0}^{\infty}\frac{(2^{j+1}A)^{n}}{(2^jA)^{n}}\frac{(1+2^{j+1}A/\rho(0))^{\theta}}{(1+2^jA/\rho(0))^{\frac{N}{k_0+1}}}\Big(\frac{1}{(1+\frac{R}{\rho(0)})^\theta R^n}\\
&\qquad\displaystyle\times\int_{B(0,R)}w(x)dx\Big)^{1/p}\Big(\frac{1}{(1+\frac{R}{\rho(0)})^\theta R^n}\int_{B(0,R)}w^{-p'/p}(y)dy\Big)^{1/p'}\\
&\quad\leq C\displaystyle\sum_{j=0}^{\infty}\frac{(1+2^{j}A/\rho(0))^{\theta}}{(1+2^jA/\rho(0))^{\frac{N}{k_0+1}}}\\
&\quad\leq  CA^{-(\frac{N}{k_0+1}-\theta)}.
\end{array}
\end{equation*}
Therefore, we obtain 
\begin{equation*}
\lim_{A\rightarrow\infty}\int_{|x|>A}|g_{b,\gamma}f(x)|^pw(x)dx=0,
\end{equation*}
holds uniformly for $f\in F$.

It remains to show that $\rm(iii)$ of Lemma \ref{lem2.8} holds, namely, the set $\mathcal{F}$ is uniformly equicontinuous.
It suffices to verify that
\begin{equation*}
\lim\limits_{|h|\rightarrow0}\|g_{b,\gamma}(f)(h+\cdot)-g_{b,\gamma}(f)(\cdot)\|_{L^p(w)}=0,
\end{equation*}
holds uniformly for $f\in F$.

Taking $\gamma=|h|$ and $|h|\in(0,1)$. When $\sqrt{t}<|h|$, it holds that
$$\varphi(\frac{\sqrt{t}}{\gamma})=\varphi(\frac{\sqrt{t}}{|h|})=1.$$
This implies
$q_{t,\gamma}(x+h,y)=q_{t,\gamma}(x,y)=0.$
Therefore, we have

\begin{align}\label{3.14}
\begin{split}
&|g_b(f)(x+h)-g_b(f)(x)|\\
&\leq \displaystyle\Big(\int_{|h|^2}^{\infty}\Big|\int_{\mathbb{R}^n}q_{t,\gamma}(x+h,y)(b(x+h)-b(y))f(y)dy\\
&\quad-\int_{\mathbb{R}^n}q_{t,\gamma}(x,y)(b(x)-b(y))f(y)dy\Big|^2tdt\Big)^{1/2}\\
&\leq C\displaystyle \Big(\int_{|h|^2}^{\infty}\Big|\int_{\mathbb{R}^n}(q_{t,\gamma}(x+h,y)-q_{t,\gamma}(x,y))
(b(x+h)-b(y))f(y)dy\Big|^2tdt\Big)^{1/2}\\
&\quad+C\displaystyle \Big(\int_{|h|^2}^{\infty}\Big|\int_{\mathbb{R}^n}q_{t,\gamma}(x,y)(b(x+h)-b(x))f(y)dy\Big|^2tdt\Big)^{1/2}\\
&=:L_1+L_2.
\end{split}
\end{align}

For $L_2$, it holds that
\begin{equation}\label{3.15}
L_2(x)=C\displaystyle|b(x+h)-b(x)|g_\gamma(f)(x)\leq C\displaystyle |h|g(f)(x).
\end{equation}
By the $L^p(w)$-bounds of $g$-function, we have
\begin{equation}\label{3.16}
\|L_2\|_{L^p(w)}\leq C|h|\|f\|_{L^p(w)}.
\end{equation}

For ${L}_1$, by \eqref{3.5}, we get
\begin{align}\label{3.17}
\begin{split}
L_1&\leq C\displaystyle|h|^\delta\Big(\int_{|h|^2}^{\infty}\Big|\int_{\mathbb{R}^n} t^{-\frac{n+\delta}{2}-1}e^{-\frac{c|x-y|^2}{t}}\Big(1+\frac{\sqrt{t}}{\rho(x)}+\frac{\sqrt{t}}{\rho(y)}\Big)^{-N}\\
&\quad\times|b(x+h)-b(y)||f(y)|dy\Big|^2tdt\Big)^{1/2}\\
&\leq C\displaystyle |h|^\delta\Big\{\Big(\int_{|h|^2}^{1}\Big|\int_{|x-y|\geq\rho(x)} t^{-\frac{n+\delta}{2}-1}e^{-\frac{c|x-y|^2}{t}}\Big(1+\frac{\sqrt{t}}{\rho(x)}\Big)^{-N}\\
&\quad\times|b(x+h)-b(y)||f(y)|dy\Big|^2\Big)^{1/2}\\
&\quad+\displaystyle\Big(\int_{1}^{\infty}\Big|\int_{|x-y|\geq\rho(x)} t^{-\frac{n+\delta}{2}-1}e^{-\frac{c|x-y|^2}{t}}\Big(1+\frac{\sqrt{t}}{\rho(x)}\Big)^{-N}\\
&\quad\times|b(x+h)-b(y)||f(y)|dy\Big|^2\Big)^{1/2}\\
&\quad+\displaystyle\Big(\int_{|h|^2}^{1}\Big|\int_{|x-y|<\rho(x)} t^{-\frac{n+\delta}{2}-1}e^{-\frac{c|x-y|^2}{t}}\Big(1+\frac{\sqrt{t}}{\rho(x)}\Big)^{-N}\\
&\quad\times|b(x+h)-b(y)||f(y)|dy\Big|^2\Big)^{1/2}\\
&\quad+\displaystyle\Big(\int_{1}^{\infty}\Big|\int_{|x-y|<\rho(x)} t^{-\frac{n+\delta}{2}-1}e^{-\frac{c|x-y|^2}{t}}\Big(1+\frac{\sqrt{t}}{\rho(x)}\Big)^{-N}\\
&\quad\times|b(x+h)-b(y)||f(y)|dy\Big|^2\Big)^{1/2}\Big\}\\
&=:C|h|^\delta\Big\{L_{11}+L_{12}+L_{13}+L_{14}\Big\}.
\end{split}
\end{align}
For $L_{11}$, if $t<1$, then $t^{-\delta/2}<t^{-1/2}$.  If $|x-y|<2^k\rho(x)$, $k=1,2,\cdots$, and $|h|\leq\sqrt{t}$, then
$$|b(x+h)-b(x)|\leq C(2^k\rho(x)+\sqrt{t}).$$
Using \eqref{3.13} and splitting the integrand domain into annuli, we obtain 
\begin{align*}
&\displaystyle\Big(1+\frac{\sqrt{t}}{\rho(x)}\Big)^{-N}\int_{|x-y|\geq\rho(x)} t^{-\frac{n+\delta}{2}-1}e^{-\frac{c|x-y|^2}{t}}|b(x+h)-b(y)||f(y)|dy\\
&\leq C\displaystyle\Big(1+\frac{\sqrt{t}}{\rho(x)}\Big)^{-N}t^{\frac{M-n-1}{2}-1}\sum_{k=1}^{\infty}\int_{2^{k-1}\rho(x)\leq|x-y|<2^k\rho(x)}\frac{|f(y)|}{|x-y|^M}\\
&\quad\times|b(x+h)-b(y)|dy\\
&\leq C\displaystyle\Big(1+\frac{\sqrt{t}}{\rho(x)}\Big)^{-N}t^{\frac{M-n-1}{2}-1}\sum_{k=1}^{\infty}\frac{2^k\rho(x)+\sqrt{t}}{(2^k\rho(x))^M}\int_{|x-y|<2^k\rho(x)}|f(y)|dy\\
&\leq C\displaystyle\Big(1+\frac{\sqrt{t}}{\rho(x)}\Big)^{-N}t^{-1}\Big\{\Big(\frac{\sqrt{t}}{\rho(x)}\Big)^{M-n-1}\sum_{k=1}^{\infty}\frac{2^{-k(M-n-1-\theta\eta)}}{2^{k\theta\eta}(2^k\rho(x))^n}\\
&\quad+\displaystyle\Big(\frac{\sqrt{t}}{\rho(x)}\Big)^{M-n}\sum_{k=1}^{\infty}\frac{2^{-k(M-n-\theta\eta)}}{2^{k\theta\eta}(2^k\rho(x))^n}\Big\}\int_{|x-y|<2^k\rho(x)}|f(y)|dy\\
&\leq C\displaystyle t^{-1}\Big(1+\frac{\sqrt{t}}{\rho(x)}\Big)^{-N+1}\Big(\frac{\sqrt{t}}{\rho(x)}\Big)^{M-n-1}M_{V,\eta}f(x).
\end{align*}
Choosing $M$ and $N$ such that $N>M>n+1+\theta\eta$, then it yields that
\begin{align}\label{3.18}
\begin{split}
L_{11}&\leq C\Big(\int_{0}^{1}\Big(1+\frac{\sqrt{t}}{\rho(x)}\Big)^{-2(N-1)}\Big(\frac{\sqrt{t}}{\rho(x)}\Big)^{2(M-n-1)}\Big)^{1/2}M_{V,\eta}f(x)\\
&\leq CM_{V,\eta}f(x).
\end{split}
\end{align}

For $L_{12}$, if $t>1$, then $t^{-\delta/2}<1$. Hence
\begin{align*}
&\displaystyle\Big(1+\frac{\sqrt{t}}{\rho(x)}\Big)^{-N}\int_{|x-y|\geq\rho(x)} t^{-\frac{n+\delta}{2}-1}e^{-\frac{c|x-y|^2}{t}}|b(x+h)-b(y)||f(y)|dy\\
&\leq C\displaystyle\Big(1+\frac{\sqrt{t}}{\rho(x)}\Big)^{-N}t^{\frac{M-n}{2}-1}\sum_{k=1}^{\infty}\frac{1}{(2^k\rho(x))^M}\int_{|x-y|<2^k\rho(x)}|f(y)|dy\\
&\leq C\displaystyle t^{-1}\Big(\frac{\sqrt{t}}{\rho(x)}\Big)^{M-n}\Big(1+\frac{\sqrt{t}}{\rho(x)}\Big)^{-N}\sum_{k=1}^{\infty}\frac{2^{-k(M-n-\theta\eta)}}{2^{k\theta\eta}(2^k\rho(x))^n}\\
&\quad\times\int_{|x-y|<2^k\rho(x)}|f(y)|dy\\
&\leq C\displaystyle t^{-1}\Big(\frac{\sqrt{t}}{\rho(x)}\Big)^{M-n}\Big(1+\frac{\sqrt{t}}{\rho(x)}\Big)^{-N}M_{V,\eta}f(x).
\end{align*}
Choosing $M$ and $N$ such that $N>M>n+\theta\eta$, we have
\begin{align}\label{3.19}
\begin{split}
L_{12}\leq &C\Big(\int_{1}^{\infty}\Big(\frac{\sqrt{t}}{\rho(x)}\Big)^{2(M-n)}\Big(1+\frac{\sqrt{t}}{\rho(x)}\Big)^{-2N}\Big)^{1/2}M_{V,\eta}f(x)\\
&\leq CM_{V,\eta}f(x).
\end{split}
\end{align}

For $L_{13}$, decompose $L_{13}$ into the following three parts
\begin{align}\label{3.20}
\begin{split}
L_{13}&\leq C\displaystyle\Big(\int_{|h|^2}^{\rho(x)^2}\Big(\int_{|x-y|<\sqrt{t}} t^{-\frac{n+\delta}{2}-1}e^{-\frac{c|x-y|^2}{t}}\Big(1+\frac{\sqrt{t}}{\rho(x)}\Big)^{-N}\\
&\quad\times|b(x+h)-b(y)||f(y)|dy\Big)^2tdt\Big)^{1/2}\\
&\quad+C\displaystyle\Big(\int_{|h|^2}^{\rho(x)^2}\Big(\int_{\sqrt{t}\leq|x-y|<\rho(x)} t^{-\frac{n+\delta}{2}-1}e^{-\frac{c|x-y|^2}{t}}\Big(1+\frac{\sqrt{t}}{\rho(x)}\Big)^{-N}\\
&\quad\times|b(x+h)-b(y)||f(y)|dy\Big)^2tdt\Big)^{1/2}\\
&\quad+C\displaystyle\Big(\int_{\rho(x)^2}^{1}\Big(\int_{|x-y|<\rho(x)} t^{-\frac{n+\delta}{2}-1}e^{-\frac{c|x-y|^2}{t}}\Big(1+\frac{\sqrt{t}}{\rho(x)}\Big)^{-N}\\
&\quad\times|b(x+h)-b(y)||f(y)|dy\Big)^2tdt\Big)^{1/2}\\
&=:L^1_{13}+L^2_{13}+L^3_{13}.
\end{split}
\end{align}
If $|x-y|<\sqrt{t}$, $t<\rho(x)^2<1$ and $|h|\leq\sqrt{t}$, then $|b(x+h)-b(x)|\leq C\sqrt{t}.$
This leads to
\begin{align}\label{3.21}
\begin{split}
L^1_{13}&\leq C \displaystyle\Big(\int_{|h|^2}^{\rho(x)^2}\Big(\int_{|x-y|<\sqrt{t}} t^{-\frac{n+\delta+1}{2}}|f(y)|dy\Big)^2tdt\Big)^{1/2}\\
&\leq C \displaystyle\Big(\int_{|h|^2}^{\rho(x)^2}t^{-\delta}\Big(t^{-\frac{n}{2}}\int_{|x-y|<\sqrt{t}} |f(y)|dy\Big)^2dt\Big)^{1/2}\\
&\leq C \displaystyle\Big(\int_{0}^{\rho(x)^2}t^{-\delta}dt\Big)^{1/2}M_{V,\eta}f(x)\\
&\leq CM_{V,\eta}f(x).
\end{split}
\end{align}
If $|x-y|<2^k\sqrt{t}$, $k=1,\cdots,n$, $t<\rho(x)^2<1$ and $|h|\leq\sqrt{t}$, then $$|b(x+h)-b(x)|\leq C2^k\sqrt{t},\ \ \ \ \  \frac{\sqrt{t}}{\rho(x)}<1.$$
Choosing $M, N>0$ such that $N>M>n+1+\theta\eta$, and applying \eqref{3.13} again, we obtain
\begin{align}\label{3.22}
\begin{split}
L^2_{13}&\leq C \displaystyle\Big(\int_{|h|^2}^{\rho(x)^2}\Big(t^{\frac{M-n-\delta}{2}-1}\sum_{k=1}^{\infty}\frac{2^k\sqrt{t}}{(2^k\sqrt{t)^M}}\int_{2^{k-1}\sqrt{t}\leq|x-y|<2^k\sqrt{t}}|f(y)|dy\Big)^2tdt\Big)^{1/2}\\
&\leq C \displaystyle\Big(\int_{|h|^2}^{\rho(x)^2}t^{-\delta}\Big(\sum_{k=1}^{\infty}\frac{2^{-k(M-n-1-\theta\eta)}}{2^{k\theta\eta}(2^k\sqrt{t})^n}\int_{|x-y|<2^k\sqrt{t}} |f(y)|dy\Big)^2dt\Big)^{1/2}\\
&\leq C \displaystyle\Big(\int_{0}^{\rho(x)^2}t^{-\delta}dt\Big)^{1/2}M_{V,\eta}f(x)\\
&\leq CM_{V,\eta}f(x).
\end{split}
\end{align}
If $|x-y|<\rho(x)$, $\rho(x)^2<t<1$ and $|h|\leq\sqrt{t}$, then $|b(x+h)-b(x)|\leq C\sqrt{t}$ and $ \frac{\sqrt{t}}{\rho(x)}<1.$
Therefore
\begin{align}\label{3.23}
\begin{split}
L^3_{13}&\leq C \displaystyle\Big(\int_{\rho(x)^2}^{1}\Big(\int_{|x-y|<\rho(x)} t^{-\frac{n+\delta+1}{2}}|f(y)|dy\Big)^2tdt\Big)^{1/2}\\
&\leq C \displaystyle\Big(\int_{\rho(x)^2}^{1}t^{-\delta}\Big(t^{-\frac{n}{2}}\int_{|x-y|<\rho(x)} |f(y)|dy\Big)^2dt\Big)^{1/2}\\
&\leq C \displaystyle\Big(\int_{\rho(x)^2}^{1}t^{-\delta}\Big(\rho(x)^{-n}\int_{|x-y|<\rho(x)} |f(y)|dy\Big)^2dt\Big)^{1/2}\\
&\leq C \displaystyle\Big(\int_{0}^{1}t^{-\delta}dt\Big)^{1/2}M_{V,\eta}f(x)\\
&\leq CM_{V,\eta}f(x).
\end{split}
\end{align}
This inequality, combining with \eqref{3.20}, \eqref{3.21} and \eqref{3.22} yields that
\begin{equation}\label{3.24}
L_{13}\leq CM_{V,\eta}f(x).
\end{equation}

For $L_{14}$, decompose it into the following three parts
\begin{align}\label{3.25}
\begin{split}
L_{14}&\leq C\displaystyle\Big(\int_{1}^{\rho(x)^2}\Big(\int_{|x-y|<\sqrt{t}} t^{-\frac{n+\delta}{2}-1}e^{-\frac{c|x-y|^2}{t}}\Big(1+\frac{\sqrt{t}}{\rho(x)}\Big)^{-N}\\
&\quad\times|b(x+h)-b(y)||f(y)|dy\Big)^2tdt\Big)^{1/2}\\
&\quad+C\displaystyle\Big(\int_{1}^{\rho(x)^2}\Big(\int_{\sqrt{t}\leq|x-y|<\rho(x)} t^{-\frac{n+\delta}{2}-1}e^{-\frac{c|x-y|^2}{t}}\Big(1+\frac{\sqrt{t}}{\rho(x)}\Big)^{-N}\\
&\quad\times|b(x+h)-b(y)||f(y)|dy\Big)^2tdt\Big)^{1/2}\\
&\quad+C\displaystyle\Big(\int_{\rho(x)^2}^{\infty}\Big(\int_{|x-y|<\rho(x)} t^{-\frac{n+\delta}{2}-1}e^{-\frac{c|x-y|^2}{t}}\Big(1+\frac{\sqrt{t}}{\rho(x)}\Big)^{-N}\\
&\quad\times|b(x+h)-b(y)||f(y)|dy\Big)^2tdt\Big)^{1/2}\\
&=:L^1_{14}+L^2_{14}+L^3_{14}.
\end{split}
\end{align}
It is easy to see that
\begin{align}\label{3.26}
\begin{split}
L^1_{14}&\leq C \displaystyle\Big(\int_{1}^{\rho(x)^2}\Big(\int_{|x-y|<\sqrt{t}} t^{-\frac{n+\delta}{2}-1}|f(y)|dy\Big)^2tdt\Big)^{1/2}\\
&\leq C \displaystyle\Big(\int_{1}^{\rho(x)^2}t^{-(1+\delta)}\Big(t^{-\frac{n}{2}}\int_{|x-y|<\sqrt{t}} |f(y)|dy\Big)^2dt\Big)^{1/2}\\
&\leq C \displaystyle\Big(\int_{1}^{\rho(x)^2}t^{-(1+\delta)}dt\Big)^{1/2}M_{V,\eta}f(x)\\
&\leq CM_{V,\eta}f(x).
\end{split}
\end{align}
Using \eqref{3.13} with  $M>n+\theta\eta$, we get
\begin{align}\label{3.27}
\begin{split}
L^2_{14}&\leq C \displaystyle\Big(\int_{1}^{\rho(x)^2}t^{-(1+\delta)}\Big(t^{\frac{M-n}{2}}\sum_{k=1}^{\infty}\int_{2^{k-1}\sqrt{t}\leq|x-y|<2^k\sqrt{t}}\frac{|f(y)|}{|x-y|^M}dy\Big)^2dt\Big)^{1/2}\\
&\leq C \displaystyle\Big(\int_{1}^{\rho(x)^2}t^{-(1+\delta)}\Big(\sum_{k=1}^{\infty}\frac{2^{-k(M-n-\theta\eta)}}{2^{k\theta\eta}(2^k\sqrt{t})^n}\int_{|x-y|<2^k\sqrt{t}} |f(y)|dy\Big)^2dt\Big)^{1/2}\\
&\leq C \displaystyle\Big(\int_{1}^{\rho(x)^2}t^{-(1+\delta)}dt\Big)^{1/2}M_{V,\eta}f(x)\\
&\leq CM_{V,\eta}f(x).
\end{split}
\end{align}
If $t\geq 1$, then $t^{-\delta/2}\leq 1$. Taking $N=1$, we have $(\sqrt{t}/\rho(x))^{-2N}=\rho(x)^2/t$, therefore
\begin{align}\label{3.28}
\begin{split}
L^3_{14}&\leq C \displaystyle\Big(\int_{\rho(x)^2}^{\infty}\Big(\frac{\rho(x)}{t}\Big)^2\Big(\int_{|x-y|<\rho(x)} t^{-\frac{n}{2}}|f(y)|dy\Big)^2dt\Big)^{1/2}\\
&\leq C \displaystyle\Big(\int_{\rho(x)^2}^{\infty}\Big(\frac{\rho(x)}{t}\Big)^2\Big(\rho(x)^{-n}\int_{|x-y|<\rho(x)} |f(y)|dy\Big)^2dt\Big)^{1/2}\\
&\leq C \displaystyle\Big(\int_{\rho(x)^2}^{\infty}\Big(\frac{\rho(x)}{t}\Big)^2dt\Big)^{1/2}M_{V,\eta}f(x)\\
&\leq CM_{V,\eta}f(x).
\end{split}
\end{align}
Sum up \eqref{3.25}-\eqref{3.28} in all, we have
\begin{equation}\label{3.29}
L_{14}\leq CM_{V,\eta}f(x).
\end{equation}
Inequality \eqref{3.29} together with \eqref{3.24}, \eqref{3.19}, \eqref{3.18} and \eqref{3.17} gives that
\begin{equation}\label{3.30}
L_1\leq C|h|^\delta M_{V,\eta}f(x).
\end{equation}
By Lemma \ref{lem2.4} for any $p'\leq\eta<\infty$, we get
\begin{equation}\label{3.31}
\|L_1\|_{L^p(w)}\leq C|h|^\delta\|M_{V,\eta}f\|_{L^p(w)}\leq C|h|^\delta\|f\|_{L^p(w)}.
\end{equation}
which together with \eqref{3.14} and \eqref{3.15} yields that
$$
\|g_bf(h+\cdot)-g_bf(\cdot)\|_{L^p(w)}\leq C(|h|+|h|^\delta)\|f\|_{L^p(w)}.
$$
This completes the proof of Theorem \ref{thm1.1}.
\end{proof}

\medskip
\begin{proof}[Proof of Theorem \ref{thm1.2}]
We will complete the proof of Theorem \ref{thm1.2} by slightly modifying the proof in \cite{WX}. We only give the proof of the modified part here.
Since the condition  (\ref{1.8}) was used only in the step to prove for $w\in A_p^{\rho,\theta}$,
	\begin{equation}\label{3.32}
	\lim_{A\rightarrow\infty}\int_{|x|>A}|\mathcal{T}^*_{b,\gamma}f(x)|^pw(x)dx=0,
	\end{equation}
	whenever $f\in F$.  Therefore it suffices to prove (\ref{3.32}) without the condition (\ref{1.8}). As the proof of Theorem \ref{thm1.1}, it suffices to show that for any $N>0$, there is a constant $C_N>0$ so that
	\begin{equation}\label{3.33}
	\sup_{t>0}k_t(x,y)\leq \frac{C_N}{|x-y|^{n}}\Big(1+\frac{|x-y|}{\rho(x)}\Big)^{-N}.
	\end{equation}
	In fact, we may control the left side of \eqref{3.33} by
	\begin{align*}
	\begin{split}
	\sup_{t>0}k_t(x,y)&\leq C_N\sup_{t>0} t^{-\frac{n}{2}}e^{-\frac{|x-y|^2}{5t}}\Big(1+\frac{\sqrt{t}}{\rho(x)}+\frac{\sqrt{t}}{\rho(y)}\Big)^{-N}\\
	&\leq C_N\sup_{\sqrt{t}\geq|x-y|}t^{-\frac{n}{2}}e^{-\frac{|x-y|^2}{5t}}\Big(1+\frac{\sqrt{t}}{\rho(x)}\Big)^{-N}\\
	&\quad+C_N\sup_{\sqrt{t}<|x-y|}t^{-\frac{n}{2}}e^{-\frac{|x-y|^2}{5t}}\Big(1+\frac{\sqrt{t}}{\rho(x)}\Big)^{-N}\\
	&=I+J.
	\end{split}
	\end{align*}
	We need to consider the contributions of  $I$ and $J$. For $I$, using the estimate $e^{-s}\leq\frac{C}{s^{n/2}}$, we have
	$$I\leq \sup_{\sqrt{t}\geq|x-y|}\frac{C_N}{|x-y|^n}\Big(1+\frac{\sqrt{t}}{\rho(x)}\Big)^{-N}\leq\frac{C_N}{|x-y|^{n}}\Big(1+\frac{|x-y|}{\rho(x)}\Big)^{-N}.\\$$
	For $J$, using the estimate $e^{-s}\leq\frac{C}{s^{n/2+N/2}}$, it follows that
	\begin{align*}
	\begin{split}
	J&\leq C_N\sup_{\sqrt{t}<|x-y|} t^{-\frac{n}{2}}\Big(\frac{t}{|x-y|^2}\Big)^{\frac{n}{2}+\frac{N}{2}}\Big(1+\frac{\sqrt{t}}{\rho(x)}\Big)^{-N}\\
	&= \sup_{\sqrt{t}<|x-y|}\frac{C_N}{|x-y|^n}\Big(\frac{\sqrt{t}}{|x-y|}\Big)^{N}\Big(1+\frac{\sqrt{t}}{\rho(x)}\Big)^{-N}\\
	&\leq \sup_{\sqrt{t}<|x-y|}\frac{C_N}{|x-y|^n}\Big(\frac{\sqrt{t}+\rho(x)}{|x-y|+\rho(x)}\Big)^{N}\Big(\frac{\sqrt{t}+\rho(x)}{\rho(x)}\Big)^{-N}\\
	&=\frac{C_N}{|x-y|^{n}}\Big(1+\frac{|x-y|}{\rho(x)}\Big)^{-N}.
	\end{split}
	\end{align*}
	Therefore
	\begin{equation*}
	\sup_{t>0}k_t(x,y)\leq \frac{C_N}{|x-y|^{n}}\Big(1+\frac{|x-y|}{\rho(x)}\Big)^{-N},
	\end{equation*}
	holds for any $N>0$. Then, similar to the proof of Theorem \ref{thm1.1}, we may show that  \eqref{3.32} holds. This finishes the proof of Theorem \ref{thm1.2}.

\end{proof}

\end{document}